\documentclass[twoside,11pt]{amsart} \usepackage{epsfig}
\usepackage{latexsym} 
\usepackage{amsmath} \usepackage{amssymb}



\newtheorem{theorem}{Theorem}[section]
\newtheorem{lemma}[theorem]{Lemma}

\newdimen\epsfxsize
\newdimen\epsfysize

\newcommand {\gap}     {\makebox[0.075 in]{}}   
\newcommand {\biggap}     {\makebox[0.2 in]{}}   
   
\newcommand{\lil}   {\scriptstyle }

\newcommand {\pml}[1]  {{#1}^{\#}}

\newcommand{\pgap}   {{\mathcal G}}

\begin{document}

\title{On the counts $\Phi(x,p)$ of $p$-rough numbers}

\date{12 Feb 2024; first draft 14 Aug 2023}

\author{Fred B. Holt}
\address{fbholt62@gmail.com; https://www.primegaps.info}

\maketitle

\begin{abstract}
The $p$-rough numbers are those numbers {\em all} of whose prime factors are greater than $p$.
These are the numbers that remain as candidate primes after Eratosthenes sieve has been advanced
up through the prime $p$.

We denote by $\Phi(x,p)$ the count of $p$-rough numbers up to and including $x$.  The asymptotic behavior
of $\Phi(x,p)$ has been described by Tenenbaum, Buchstab, et al.

From our studies of Eratosthenes sieve as a discrete dynamic system, we identify symmetries for $\Phi(x,p)$ for
fixed $p$.  These symmetries
are most easily seen in the derived function 
$${\Delta \Phi(x,p) = \Phi(x,p) - \frac{\phi(\pml{p})}{\pml{p}} x},$$
which measures the difference between $\Phi(x,p)$ and the line with slope $\frac{\phi(\pml{p})}{\pml{p}}$.
For fixed $p$, the function $\Delta \Phi(x,p)$ has a translational symmetry of period $\pml{p}$, 
$$ \Delta \Phi(x+\pml{p},p) \; = \; \Delta \Phi(x,p) \biggap {\rm with} \gap \Phi(k \cdot \pml{p},p)=0,$$
and a rotational symmetry around the midpoints $\tilde{x}_k = (k-\frac{1}{2})\pml{p}$
$$ \Delta \Phi(\tilde{x}_k+x, p) \; = \; - \Delta \Phi(\tilde{x}_k-x,p) \biggap {\rm for} \; 0 \le x \le \frac{\pml{p}}{2}.$$

Previous work on $p$-rough numbers that estimates the error in $\Phi(x,p)$ relative to the line $y=\frac{x}{\ln p}$
misses the line of symmetry.  These estimates take the lines of symmetry to their limit while leaving the residual counts
behind.  These other estimates primarily measure the growing drift between the surrogate line
$y=\frac{x}{\ln p}$ and the true line of symmetry $y=\frac{\phi(\pml{p})}{\pml{p}} x$.  In this sense, they are 
estimates of the error in Merten's Third Theorem.
\end{abstract}

\section{Setting}

$\Phi(x,p)$ is the number of $p$-rough numbers $\gamma \le x$.  A number $\gamma$ is $p$-rough iff all of the 
prime factors of $\gamma$ are greater than $p$.  We include $1$ as a $p$-rough number.  The $p$-rough numbers
are exactly the unit $1$ and the candidate primes that are left after Eratosthenes sieve has been advanced
up through the prime $p$.

Tenenbaum has shown that
$$ \Phi(x,p) \; = \; \frac{x}{\ln p} \left( \omega(u) + O(\frac{1}{\ln p}) \right) $$
where $\omega(u)$ is the Buchstab function \cite{ChGold, FrGrOsc, FanPom}. 
One central contribution of this paper is to establish $y=\frac{\phi(\pml{p})}{\pml{p}}x$ as a line of symmetry
for $\Phi(x,p)$, so that estimates against other lines, such as $y=\frac{x}{\ln p}$, will eventually be primarily measuring the
drift between the line of symmetry and any surrogate line. 

Using the line of symmetry for the analysis of $\Phi(x,p)$, we can introduce the function
$$ \Delta \Phi(x,p) \; = \; \Phi(x,p) - \frac{\phi(\pml{p})}{\pml{p}} x$$
which nicely decouples the linear growth of $\Phi(x,p)$ from its deviations around its line of symmetry.

\vspace{0.125in}

Let $p$ be a prime and $q$ the next larger prime.
There is a cycle of gaps $\pgap(\pml{p})$ among the candidate primes remaining after Eratosthenes sieve has been 
advanced through the prime $p$.
The cycle $\pgap(\pml{p})$ has length $\phi(\pml{p})$ and span $\pml{p}$.
The cycle is symmetric, $g_k = g_{\phi(\pml{p})-k}$, with 
$$g_1 = g_{\phi(\pml{p})-1} = q-1 \biggap {\rm and} \biggap g_{\phi(\pml{p})}=2.$$
These cycles of gaps $\pgap(\pml{p})$ are studied in \cite{FBHSFU, FBHPatterns}.

For example, the cycle $\pgap(\pml{5})$ has length (number of gaps) $\phi(\pml{5})=8$ and span
(sum of gaps) $\pml{5}=30$.
$$ \pgap(\pml{5}) \biggap = \biggap 6 \; 4 \; 2 \; 4 \; 2 \; 4 \; 6 \; 2.$$
Starting with $1$, these gaps separate the $5$-rough numbers
$$ 1, \; 7, \; 11, \; 13, \; 17, \; 19, \; 23, \; 29, \; 31, \; 37, \; 41, \; 43, \; 47, \; 49, \; 53, \ldots $$

The gaps in the cycle $\pgap(\pml{p})$ separate the $p$-rough numbers.
So features of the cycle $\pgap(\pml{p})$ govern the distribution of the $p$-rough numbers and are thereby
reflected in $\Phi(x,p)$ for fixed $p$.  Specifically,  since $\pgap(\pml{p})$ is a cycle of
length $\phi(\pml{p})$ and span $\pml{p}$, there is a {\em translational symmetry}
$$ \Phi(x + \pml{p}, p)  = \Phi(x,p) + \phi(\pml{p}).$$
That is, for any $x > 0$, between $x$ and $x+\pml{p}$ there is one complete cycle of $\pgap(\pml{p})$, and
thus $\phi(p)$ additional $p$-rough numbers in this interval.

\begin{figure}[htb]
\centering
\includegraphics[width=5in]{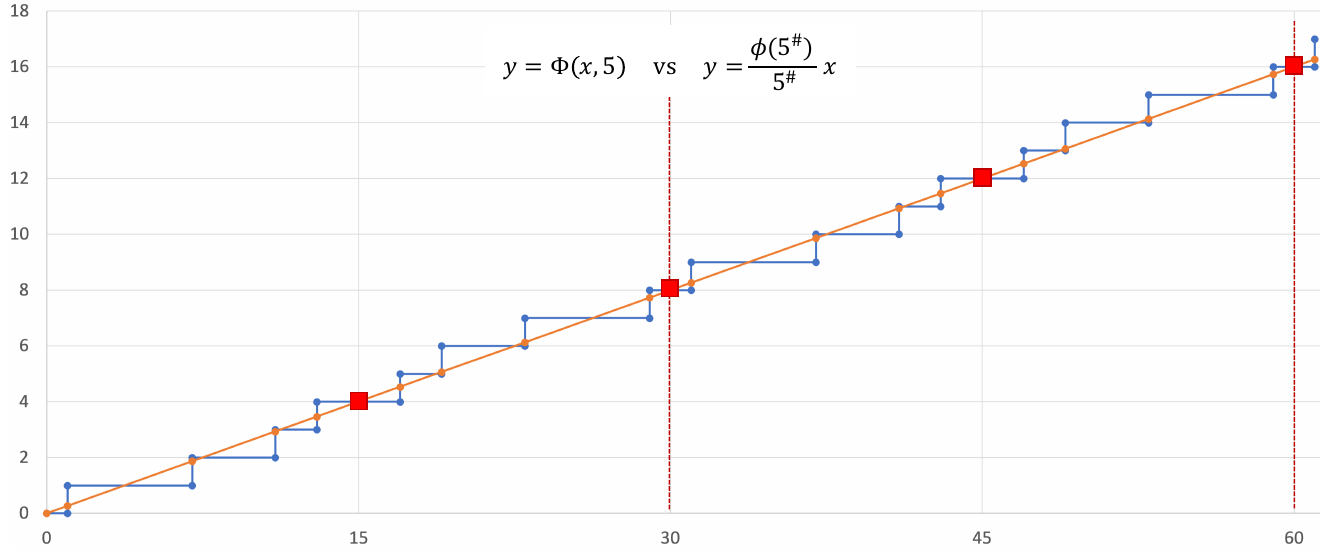}
\caption{\label{5roughPhiFig}  ${y=\Phi(x,5)}$ and ${y = \frac{\phi(\pml{5})}{\pml{5}} x }$
are shown over two periods. The cycle $\pgap(\pml{5})$ has length $8$ and span $30$.
We mark the waypoints at the ends of the periods and at the midpoints, and the periods of span $30$ are delimited, to
highlight the translational and rotational symmetries of $\Phi(x,5)$.}
\end{figure}

{\em Waypoints at ends of periods.}  From the translational symmetry and the initial condition $\Phi(0,p)=0$, we have the set
of waypoints $x_k = k \cdot \pml{p}$ for which we have the set values
$$ \Phi(x_k,p) \; = \; \Phi(k \cdot \pml{p}, p) \; = \; k \cdot \phi(\pml{p}) \biggap {\rm for~all} \gap k \ge 1.$$
These waypoints lie along the line
$$ y \; = \; \frac{\phi(\pml{p})}{\pml{p}} x.$$
That is, the graph of $y = \Phi(x,p)$ crosses this line $ y = \frac{\phi(\pml{p})}{\pml{p}} x$ regularly at these waypoints.

Moreover by the translational symmetry of $\Phi(x,p)$, the behavior of $\Phi(x,p)$ is completely determined by its behavior
in the first period of the cycle $\pgap(\pml{p})$, between $x_0 = 0$ and $x_1 = \pml{p}$.  The deviations of $\Phi(x,p)$
from the line $ y = \frac{\phi(\pml{p})}{\pml{p}} x$ in the first period of $\pgap(\pml{p})$ are repeated over every subsequent period
of the cycle.
In Figure~\ref{5roughPhiFig} we see an example for $p=5$ of the waypoints and the translational symmetry of $\Phi(x,5)$ along
the line $y=\frac{4}{15} x$.

\vspace{0.05in}

{\em Rotational symmetry of $\Phi(x,p)$.}  The cycle $\pgap(\pml{p})$ also has a reflective symmetry that corresponds to 
a rotational symmetry in the graph of ${y = \Phi(x,p)}$.  In the cycle $\pgap(\pml{p})$ we have the symmetry
$$ g_{\phi(\pml{p})-j} \; = \; g_j $$
with $g_1 = g_{\phi(\pml{p})-1} = q-1$, one less than the next prime $q$, and $g_{\phi(\pml{p})} = 2$.

From this symmetry in the cycle of gaps $\pgap(\pml{p})$ we get the corresponding symmetry in $\Phi(x,p)$ of
\begin{eqnarray*}
 \Phi(x_1 - x, p) & = & \phi(\pml{p}) - \Phi(x,p) \\
{\rm and} \biggap \Phi(x_k - x, p) & = & k \cdot \phi(\pml{p}) - \Phi(x,p) \biggap {\rm for} \gap 0 \le x \le \frac{\pml{p}}{2}
\end{eqnarray*}

This shows up in the graph of $\Phi(x,p)$ as a rotational symmetry around an additional set of waypoints, the midpoint waypoints
$$ \tilde{x}_k = \left( k -\frac{1}{2} \right) \pml{p}$$
which also lie along the line $y = \frac{\phi(\pml{p})}{\pml{p}} x$.
$$ \Phi(\tilde{x}_k, p) \; = \; \left(k - \frac{1}{2}\right) \phi(\pml{p}).$$


We observe examples of the midpoint waypoints and the rotational symmetry in the graph of $\Phi(x,5)$ in Figure~\ref{5roughPhiFig}.
We also note that at the vertical rises in $\Phi(x,p)$, we need to carefully consider both the lower endpoint and upper endpoint
for the rotational symmetry to hold.  That is, we would technically write
$$ \Phi(x_k - x, p) \; = \; k \cdot \phi(\pml{p}) - \lim_{\epsilon \rightarrow 0} \Phi(x-\epsilon,p) \biggap 
{\rm for} \gap 0 \le x \le \frac{\pml{p}}{2}. $$

If $x$ is a $p$-rough number, there is a discontinuity in $\Phi(x,p)$ at $x$, and for the rotational symmetry we need to approach
the lower value of $\Phi(x^-,p)$ as we approach from the left in order to get the value $\Phi(x_k-x, p)$ where the image of
the discontinuity at $x$ is flipped upside down.

Under the rotational symmetry we see that the midpoint waypoints also lie along that line $y = \frac{\phi(\pml{p})}{\pml{p}}x$,
and that the deviations of $\Phi(x,p)$ from this line are rotational and translational images of the deviations in the first half of the first period of
the cycle $\pgap(\pml{p})$.

\section{The function $\Delta \Phi(x,p) = \Phi(x,p) - \frac{\phi(\pml{p})}{\pml{p}} \cdot x$.}
There is a certain simplicity in working with deviations of $\Phi(x,p)$ from the line of symmetry
${y = \frac{\phi(\pml{p})}{\pml{p}} x}$.  The function
\begin{eqnarray*}
 \Delta \Phi(x,p) & = & \Phi(x,p) - \frac{\phi(\pml{p})}{\pml{p}} \cdot x \\
  & = & \Phi(x,p) - \frac{1}{\mu} \cdot x
\end{eqnarray*}
measures these deviations of $\Phi(x,p)$ from the line of symmetry through the waypoints.

In this definition of $\Delta \Phi(x,p)$ we note that the slope of the line of symmetry is the reciprocal of the average gap
size $\mu$ in the cycle $\pgap(\pml{p})$.  There are $\phi(\pml{p})$ gaps in $\pgap(\pml{p})$ of total sum $\pml{p}$, so the
average gap size ${\mu = \mu(p) = \frac{\pml{p}}{\phi(\pml{p})}}$.
The line ${y=\frac{1}{\mu}\cdot x}$ measures $x$ in terms of the average gap size in $\pgap(\pml{p})$. 

Using the line of symmetry $y=\frac{1}{\mu} x$, the function $\Delta \Phi(x,p)$ is bounded and periodic, with rotational
symmetries.  Any analysis that uses a different line, such as $y=\frac{x}{\ln p}$ breaks this symmetry and the bounded deviations, simply
by using the wrong line, and these analyses will ultimately be measuring the drift between the surrogate line and the line of symmetry.

\begin{figure}[htb]
\centering
\includegraphics[width=5in]{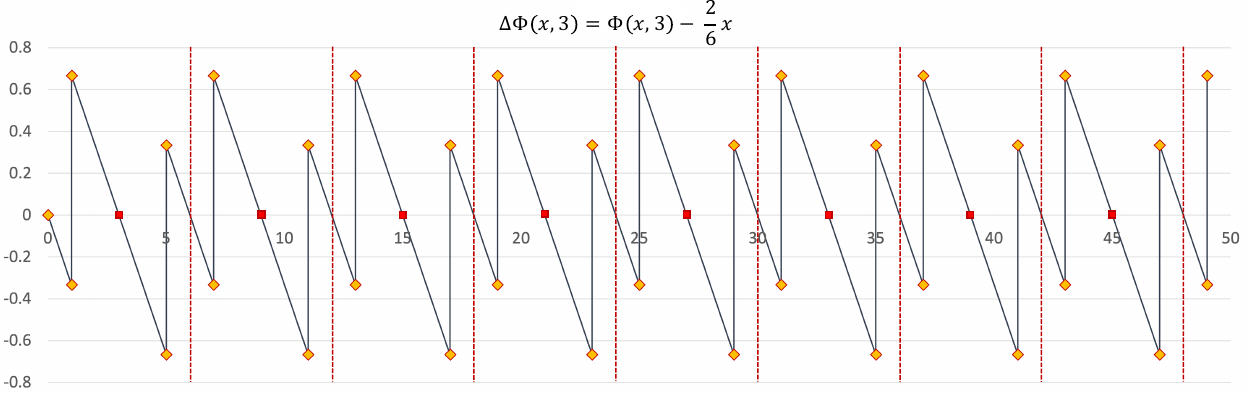}
\caption{\label{3roughFig} The equation ${\Delta \Phi(x,p) = \Phi(x,p) - \frac{1}{\mu} x }$
for ${p=3}$.  The cycle $\pgap(\pml{3})$ has length $2$ and span $6$.  The waypoints are highlighted and the 
periods of span $6$ are delimited, to highlight the symmetries.  Note the translational symmetry across periods
of the cycle, and the rotational symmetry around the waypoints $x_k$ and $\tilde{x}_k$.}
\end{figure}

The waypoints $x_k$ and $\tilde{x}_k$ all lie along this line, and the translational and rotational symmetries
of $\Phi(x,p)$ keep the values close to this line.  
So we turn our attention to the function 
$$\Delta \Phi(x,p) = \Phi(x,p) - \frac{1}{\mu} x.$$

The function $\Delta \Phi(x,p)$ has a sawtooth graph, with vertical rises of $1$ at each new $p$-rough number,
followed by parallel lines of decay of slope ${\frac{-\phi(\pml{p})}{\pml{p}} = \frac{-1}{\mu}}$.  From $\Phi(x,p)$,
the function $\Delta \Phi(x,p)$ inherits the following properties.

\begin{lemma}
Waypoints at ends of cycles:  There are waypoints at the end of each cycle ${x_k = k \cdot \pml{p}}$, 
for which
$$\Delta \Phi(x_k,p) = 0.$$
\end{lemma}

\begin{lemma} Midpoint waypoints:  There are waypoints at the middle of each cycle, ${\tilde{x}_k = (k-\frac{1}{2}) \pml{p}}$, 
for which
$$\Delta \Phi(\tilde{x}_k,p)=0.$$
\end{lemma}

\begin{lemma} Symmetries:  $\Delta \Phi$ has a translational symmetry
$$\Delta \Phi(x+\pml{p},p) = \Delta \Phi(x,p) \gap {\rm for~all} \gap x \ge 0, $$
and rotational symmetries around its waypoints:
\begin{eqnarray*}
 \Delta \Phi(x_k-x,p) &  = & - \Delta \Phi(x_k+x,p) \gap {\rm for~all} \gap x_k-x, \; x_k+x \ge 0 \\
 \Delta \Phi(\tilde{x}_k-x,p) &  = & - \Delta \Phi(\tilde{x}_k+x,p) \gap {\rm for~all} \gap \tilde{x}_k-x, \; \tilde{x}_k+x \ge 0 
\end{eqnarray*}
\end{lemma}

For the rotational symmetries we have to be careful to treat the upper and lower values along the vertical rises in the graph
of $\Delta \Phi(x,p)$.  Under the rotational symmetries the limits from the left are mapped to the limits from the right.

The graph of $\Delta \Phi(x,p)$ and thus the graph of $\Phi(x,p)$ are completely determined by the first half of the first
period of $\Delta \Phi(x,p)$, $0 \le x \le \tilde{x}_1$.

In Figure~\ref{3roughFig} we can plot $\Delta \Phi(x,3)$ over several periods of the cycle ${\pgap(\pml{3})= 4\; 2}$.
This cycle is simple enough that we can observe both the translational and rotational symmetries, and the
sawtooth structure is easily visible.

For $p=5$, the cycle $\pgap(\pml{5}) = 6 4 2 4 2 4 6 2$, of length $8$ and span $30$.  In Figure~\ref{5roughFig}
we graph ${\Delta \Phi(x,5)}$ over a few periods of the cycle, and we still readily 
see both the sawtooth structure of the graph and its symmetries.

In later figures, we show two periods of $\Delta \Phi(x,7)$ and the first period for each of
$\Delta \Phi(x,11)$, $ \Delta \Phi(x,13)$, and $ \Delta \Phi(x,17)$

The cycle $\pgap(\pml{5})$ begins to exhibit the complexities that arise in these cycles of gaps $\pgap(\pml{p})$.
So the graph of $\Delta \Phi(x,5)$ is more interesting but still simple enough.  In Figure~\ref{5roughFig} the 
midpoint waypoints are marked by red squares, and we can see the rotational symmetry between the front half of 
the cycle and the back half.  With respect to the midpoint we write the rotational symmetry:
$$ \Delta \Phi \left(\frac{\pml{p}}{2}+ x, p \right) = - \; \Delta \Phi \left(\frac{\pml{p}}{2}- x, p \right) \gap {\rm for} 
\gap 0 \le x \le \frac{\pml{p}}{2}.$$

\begin{figure}[htb]
\centering
\includegraphics[width=5in]{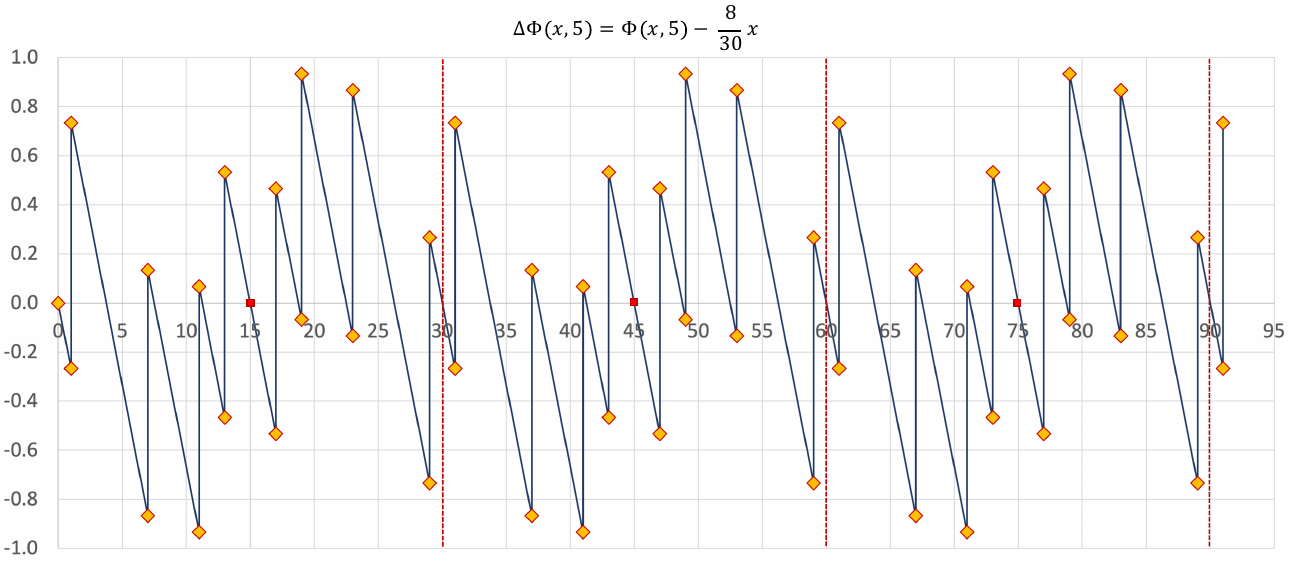}
\caption{\label{5roughFig} The equation ${\Delta \Phi(x,5) = \Phi(x,5) - \frac{1}{\mu} x }$
is shown over three periods. The cycle $\pgap(\pml{5})$ has length $8$ and span $30$.
The waypoints are highlighted and the periods of span $30$ are delimited, to highlight the symmetries.}
\end{figure}

In Figure~\ref{5roughFig} we note a couple of artifacts that will show up in all the graphs $\Delta \Phi(x,p)$.
The first vertical rise occurs at $x=1$.  The lower value here, decaying from $x=0$, is 
$$ \Delta \Phi(1^{-},p) = -\frac{\phi(\pml{p})}{\pml{p}} \; = \; - \frac{1}{\mu(p)},$$
and the upper value on this vertical edge is $1$ above the lower value.
$$ \Delta \Phi(1,p) = 1-\frac{1}{\mu}.$$
In the sawtooth, all of the vertical rises are of length $1$, incrementing the count for this next $p$-rough number.

At the end of every cycle $\pgap(\pml{p})$ there is a gap $2$ that carries us from $x_k-1$ to $x_k+1$.
By the symmetries of $\Delta \Phi(x,p)$, we must have
$$ \Delta \Phi(x_k -1, p) = \frac{1}{\mu} \biggap {\rm and} \biggap  
\Delta \Phi(x_k +1^{-}, p) = - \frac{1}{\mu}. $$

Let $q$ be the next prime larger than $p$.  The first gap in the cycle $\pgap(\pml{p})$ is $g_1=q-1$.
After the $p$-rough number $x=1$, there follows a long linear decay of length $\Delta x = q-1$,
dropping from
$$ \Delta \Phi(1,p)=1 - \frac{1}{\mu} \biggap {\rm to} \biggap 
\Delta \Phi(q^{-},p)=1 -  \frac{q}{\mu}. $$
The vertical rise of $1$ lifts the value to
$$\Delta \Phi(q,p) = 2 - \frac{q}{\mu}. $$

For every prime $p$ we know a few values for $\Delta \Phi(x,p)$, listed in Table~\ref{ValueTbl}.

\begin{figure}[htb]
\centering
\includegraphics[width=5in]{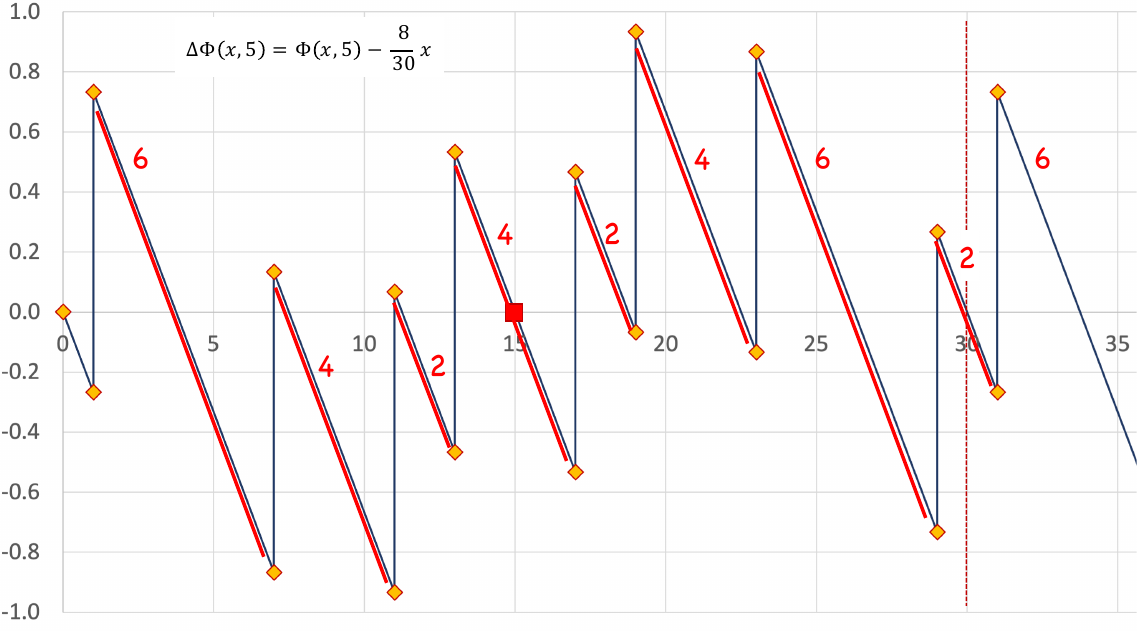}
\caption{\label{5roughZoomFig} The equation ${\Delta \Phi(x,5) = \Phi(x,5) - \frac{\phi(\pml{5})}{\pml{5}} x }$
is shown over one period.  We mark the downward segments with the corresponding gap $g$ in the
cycle $\pgap(\pml{5}) = 6 4 2 4 2 4 6 2$.}
\end{figure}

\begin{center}
\begin{table}[htb]
\renewcommand{\arraystretch}{1.75}
\begin{tabular}{|c|ccc|ccc|} \hline
$x$ & $0$ & $1$ & $q$ & $\frac{\pml{p}}{2}$ &$\frac{\pml{p}}{2}-2$ &$\frac{\pml{p}}{2}-2^j$ \\ \hline
$\Delta \Phi(x^-,p)$ & $0$ & $-\frac{1}{\mu}$ & $1-\frac{q}{\mu}$ & $0$ & $\frac{2}{\mu}-1$ & $\frac{2^j}{\mu}-j$ \\
$\Delta \Phi(x,p)$ &  $0$ & $1-\frac{1}{\mu}$ & $2-\frac{q}{\mu}$ & $0$ & $\frac{2}{\mu}$ & $\frac{2^j}{\mu}-j +1$  \\ \hline
$\hat{x}=\pml{p}-x$ & $\pml{p}$ & $\pml{p}-1$ & $\pml{p}-q$ & $\frac{\pml{p}}{2}$ & $\frac{\pml{p}}{2}+2$ &$\frac{\pml{p}}{2}+2^j$ \\
$\Delta \Phi(\hat{x}^-,p)$ & $0$ & $\frac{1}{\mu}-1$ & $\frac{q}{\mu}-2$ & $0$ & $-\frac{2}{\mu}$ & $j-1-\frac{2^j}{\mu}$ \\
$\Delta \Phi(\hat{x},p)$ & $0$ & $\frac{1}{\mu}$ & $\frac{q}{\mu}-1$ & $0$ & $1-\frac{2}{\mu}$ & $j-\frac{2^j}{\mu}$ \\ \hline
\end{tabular}
\caption{\label{ValueTbl} Known values for $\Delta \Phi(x,p)$ for any value of $p$.}
\end{table}
\end{center}

As we move into $\Delta \Phi(x,7)$ as shown in Figure~\ref{7roughFig}, the length and span of the cycle
$\pgap(\pml{p})$ makes it harder to see the fine structure, but it is still there.  The graph of 
${\Delta \Phi(x,7)}$ is still a sawtooth graph with vertical rises of $1$ at the $7$-rough numbers and parallel
linear decay between.  Two periods of $\Delta \Phi(x,7)$ are shown in Figure~\ref{7roughFig} with their
midpoints highlighted.  The graph is becoming noisy but we can still see the translational and rotational 
symmetries.

\begin{figure}[htb]
\centering
\includegraphics[width=5in]{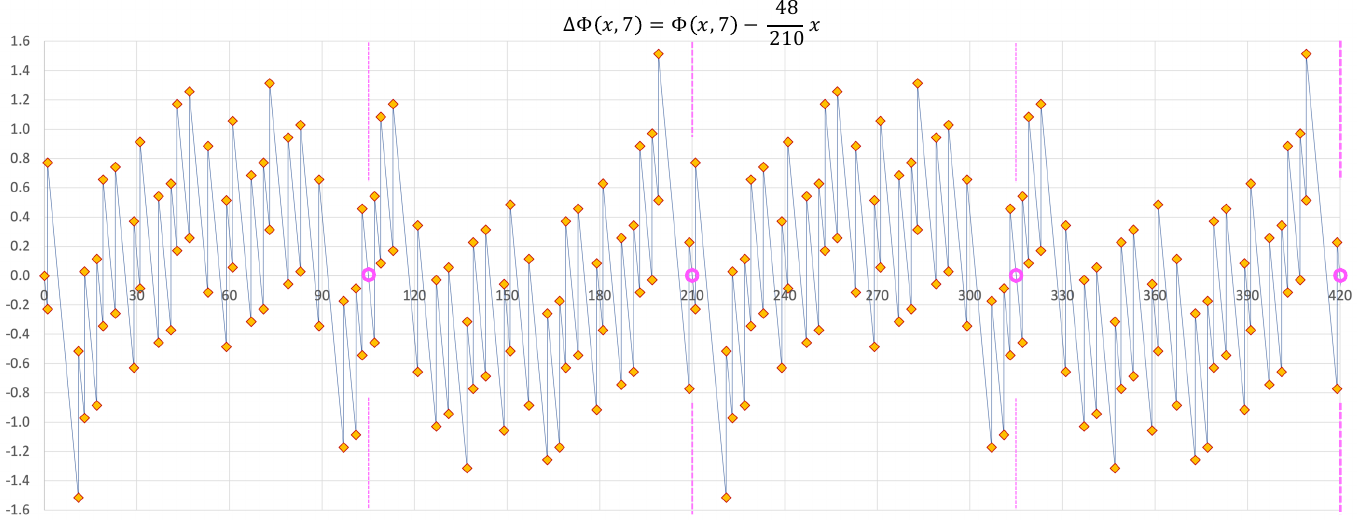}
\caption{\label{7roughFig} The equation ${\Delta \Phi(x,7) = \Phi(x,7) - \frac{\phi(\pml{7})}{\pml{7}} x }$
is shown over two periods.  The cycle $\pgap(\pml{7})$ has length $48$ and span $210$.  
The waypoints are highlighted, and the cycles of span $210$ are delimited, to highlight the symmetries.}
\end{figure}

As we increase $p$ to $11$, $13$, and $17$, it becomes harder to see the symmetries that we have identified and the
structure at the middle of the periods.  If we study the graphs in Figure~\ref{moreroughFig}, we see that
these features are still there.  In each of these graphs, we show only one period of $\Delta \Phi(x,p)$.

\begin{figure}[htb]
\centering
\includegraphics[width=5.25in]{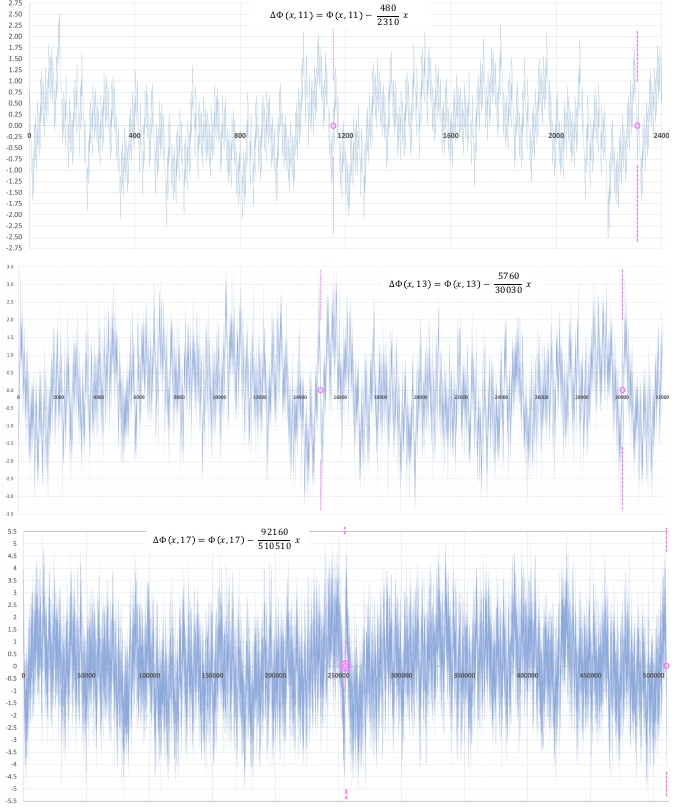}
\caption{\label{moreroughFig} The graphs for ${\Delta \Phi(x,11)}$, ${\Delta \Phi(x,13)}$, and ${\Delta \Phi(x,17)}$
are shown over one period.  The waypoints are marked and the end of the period is highlighted.  The rotational symmetry
is visible.}
\end{figure}


The upper section in Table~\ref{ValueTbl} lists the values at the end of the period.
The lower section in the table provides the values of $\Delta \Phi(x,p)$ over the constellation at the middle of the cycle $\pgap(\pml{p})$
$$ s \; = \; 2^{J} \; 2^{J-1} \; \ldots \; 4 \; 2 \; 4 \; 2 \; 4 \; \ldots \; 2^{J-1} \; 2^{J}$$
where $2^J$ is the largest power of $2$ such that $2^{J} < q$.  The last row in the table holds for all $1 \le j \le J+1$.

For $\Delta \Phi(x,13)$ and $\Delta \Phi(x,17)$ we show a closeup of the middle of the cycle in Figure~\ref{roughMidFig}.  Here we can
see both the local symmetry and the progression of segments corresponding to the middle constellation of 
$$ s = 16, \; 8 \; 4 \; 2 \; 4 \; 2 \; 4 \; 8, \; 16.$$

\begin{figure}[hbt]
\centering
\includegraphics[width=5in]{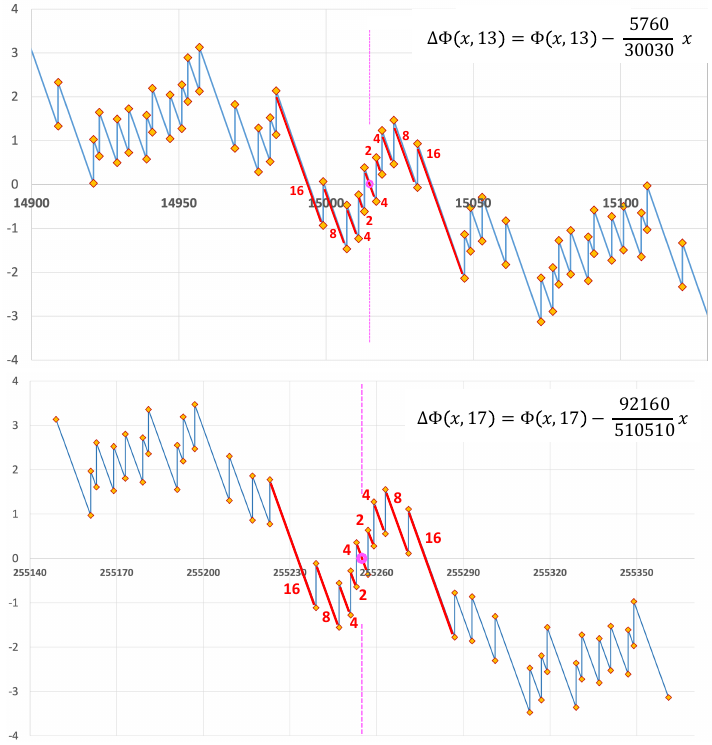}
\caption{\label{roughMidFig} The equations $\Delta \Phi(x,13)$ and $\Delta \Phi(x,17)$
are shown near the middle of the first period of the cycle $\pgap(\pml{13})$ and $\pgap(\pml{17})$.  
The waypoint at the middle of the period is highlighted.  In these close-ups we can see the sawtooth structure of the
graph, the rotational symmetry at this midpoint, and the midcycle constellation.}
\end{figure}

\section{Bounds on $\Delta \Phi(x,p)$}
From the work above, we have the following theorem.

\begin{theorem}\label{ThmPhi}
\begin{equation}
 \Phi (x,p) \; = \; \frac{\phi(\pml{p})}{\pml{p}} \cdot x \; + \; \Delta \Phi(x,p) \label{EqPhi}
 \end{equation}
in which $\Delta \Phi(x,p)$ is bounded and periodic, with period $\pml{p}$.  The bounds
$$\max_x \Delta \Phi(x,p) = \max_{p \le x \le \pml{p}} \Delta \Phi(x,p) = \limsup_{x \rightarrow \infty} \Delta \Phi(x,p)$$
and ${\min_x \Delta \Phi(x,p) = - \max_x \Delta \Phi(x,p)}.$ 
\end{theorem}

\begin{proof}
The lemmas provide the periodicity.  From the translational symmetry, we see that the behavior 
of $\Delta \Phi(x,p)$ for extremely large $x$ is exactly described by the deviations within
the first period of the cycle $\pgap(\pml{p})$.

For fixed $p$, all of the waypoints $x_k$ and $\tilde{x}_k$ lie along the line of symmetry:
$$ y = \frac{\phi(\pml{p})}{\pml{p}} x.$$
All of the deviations of $\Phi(x_k+ x, p)$ from this line between the waypoints preserve both
the periodicity and the symmetry of the cycles of gaps $\pgap(\pml{p})$.

From the translational symmetry we know that for each fixed $p$ there are values $y_{\rm max}$ and
$y_{\rm min}$ such that
\begin{eqnarray*}
y_{\rm max} & = & \max_{0 \le x \le \pml{p}} \Delta \Phi(x,p) \gap = \gap \limsup_{x \rightarrow \infty} \Delta \Phi(x,p) \\
y_{\rm min} & = & \min_{0 \le x \le \pml{p}} \Delta \Phi(x,p) \gap = \gap \liminf_{x \rightarrow \infty} \Delta \Phi(x,p) 
\end{eqnarray*}
From the rotational symmetry we know that
$$ y_{\rm min} \gap = \gap - y_{\rm max}.$$
The constellation $s_1$ that goes from $0$ down to $y_{\rm min}$ is the reflection of the constellation $-s_1$ that goes
from $y_{\rm max}$ down to $0$.  Similarly the constellation $s_2$ that goes from $y_{\rm min}$ back up to $0$ is the reflection
of the constellation $-s_2$ that goes from $0$ up to $y_{\rm max}$.
\end{proof}

Theorem~\ref{ThmPhi} provides the best decomposition of $\Phi(x,p)$.  The first term in Equation~\ref{EqPhi} is the line
of symmetry for $\Phi(x,p)$, and the residual $\Delta \Phi(x,p)$ is periodic, bounded, and symmetric.  To fully describe the behavior
of $\Delta \Phi(x,p)$, we only need to understand its behavior over the interval ${p \le x \le \pml{p}}$, or by symmetry
${p \le x \le \frac{\pml{p}}{2}}$.

\begin{figure}[htb]
\centering
\includegraphics[width=5in]{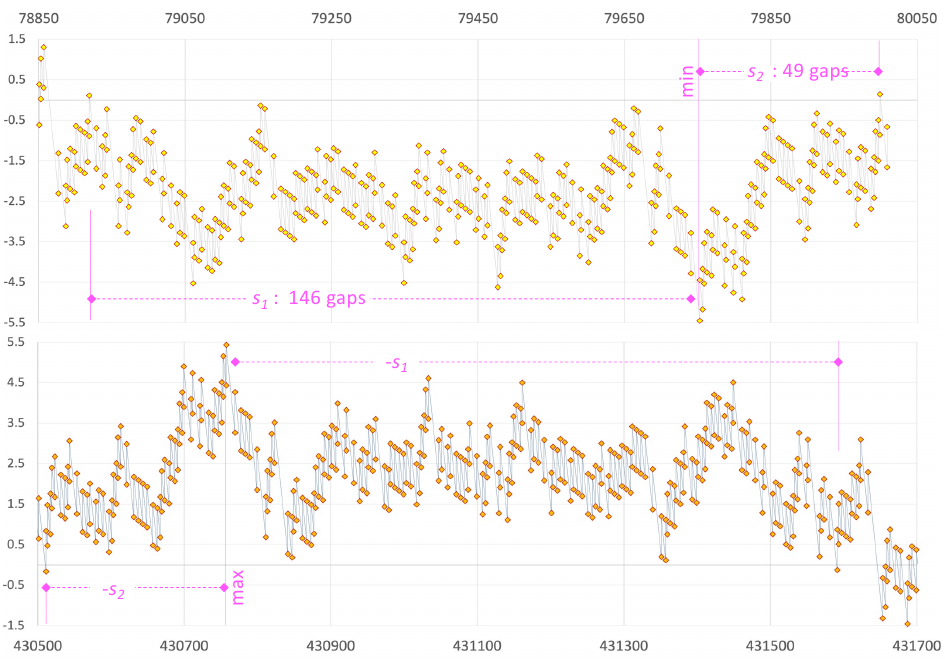}
\caption{\label{17MinMaxFig} Closeups of $\Delta \Phi(x,17)$ near the minimum value $-5.438880$ and the maximum
value $5.438880$. The constellation $s_1$ from $0$ down to the minimum has length $146$ and the constellation $s_2$ back
up to $0$ has length $49$.  By the rotational symmetry the constellation from $0$ up to the maximum is a reflection 
of $s_2$ and the constellation back down to $0$ is a reflection of $s_1$.}
\end{figure}

\vskip 0.125in

Let $s$ be a constellation in $\pgap(\pml{p})$ of length $j$ and span $|s| = \sum g$.  Then the change in $\Delta \Phi(x,p)$
from the start of $s$ to the end of $s$ is
\begin{eqnarray*}
\left. \Delta \Phi \right|_{s} (\cdot, p) & = & \Delta \Phi(x_0+ |s|, p) - \Delta \Phi(x_0,p) \\
& = & j - \frac{| s|}{\mu}
\end{eqnarray*}
for any $x_0$ that marks the start of an occurrence of $s$ in $\pgap(\pml{p})$.

For a significant rise in $\Delta \Phi$, we need a constellation whose average gap size is well below the mean, 
and for a significant fall we need a constellation whose average gap size is well above the mean.

For example, Figure~\ref{17MinMaxFig} provides closeups around the minimum and maximum in $\pgap(\pml{17})$.
The minimum $y_{\rm min}=-5.43880$, and the constellation $s_1$ that drops from $0$ down to $y_{\rm min}$ has
length $j=146$ and span ${\sum g = 834}$.  The average gap size for $s_1$ is $5.7123288$, compared to the mean value
for $\pgap(\pml{p})$: $\mu = 5.539388$ .
Similarly the constellation $s_2$ has length $j=49$ and span ${\sum g = 246}$ for an average gap size of $5.0204082$.

Table~\ref{MinMaxTbl} lists the minima and maxima for $\Delta \Phi(x,p)$ for the first few primes.

\begin{center}
\begin{table}[htb]
\renewcommand{\arraystretch}{1.25}
\begin{tabular}{|crr|rr|rr|} \hline
$p$ & $\pml{p}$ & $\phi(\pml{p})$ & $\mu(p)$ & ${\lil \max \& \min} \Delta \Phi$ & $N^+_0$ & $\frac{N^+_0}{ \phi(\pml{p})}$ \\ \hline
$5$ & $\lil 30$ &  $\lil 8$ & $\lil 3.750$ &  $\lil \pm 0.9333$ & $\lil 8$ & $\lil 100\% $ \\
$7$ & $\lil 210$ & $\lil 48$ & $\lil 4.375$ & $\lil \pm 1.5143$ & $\lil 32$ & $\lil 66.67\%$ \\
$11$ & $\lil 2310$ & $\lil 480$ & $\lil 4.813$ & $\lil \pm 2.5195$ & $\lil 262$ & $\lil  54.58\%$ \\
$13$ & $\lil 30030$ & $\lil 5760$ & $\lil 5.214$ & $\lil \pm 3.5475$ & $\lil 2216$  & $\lil  38.47\%$ \\
$17$ & $\lil 510510$ & $\lil 92160$ & $\lil 5.539$ & $\lil \pm 5.4388$ & $\lil 25948$  & $\lil  28.16\%$ \\
$19$ & $\lil 9699690$ & $\lil 1658880$ & $\lil 5.847$ & $\lil \pm 8.6592$ & $\lil 344337$ & $\lil  20.76\%$ \\
$23$ & $\lil 223092870$ & $\lil 36595360$ & $\lil 6.113$ & $\lil \pm 14.4180$ & $\lil 5438505$ & $\lil  14.90\%$ \\
$29$ & $\lil 6469693230$ & $\lil 1021870080$ & $\lil 6.331$ & $\lil \pm 20.9128$ & $\lil 109773262$ & $\lil  10.74\%$ \\
 \hline
\end{tabular}
\caption{\label{MinMaxTbl} For $\gap(\pml{p})$ and $\Delta \Phi(x,p)$ for the first few primes, 
we tabulate the mean gap size $\mu$, the min
and max values for $\Delta \Phi(x,p)$, and the number of rising zeroes for $\Delta \Phi(x,p)$.}
\end{table}
\end{center}

For single gaps the difference in $\Delta \Phi$ is
$$\left. \Delta \Phi \right|_g \; = \; 1 - \frac{g}{\mu}.$$ 
The largest rises will occur at gaps $g=2$ and the greatest drops will occur at the largest gaps in $\pgap(\pml{p})$,
connecting this exploration to the Problem of Jacobsthal.  If $\hat{g}$ is the largest gap in $\pgap(\pml{p})$, then
$$y_{\rm min} \; \le \; \frac{1}{2} \left( 1 - \frac{\hat{g}}{\mu} \right).$$
Of course $y_{\rm min}$ will be much lower than this.

Although not usually the maximal gap in $\pgap(\pml{p_k})$, we know there will always be a gap $g = 2p_{k-1}$, thus there
is a drop in $\Delta \Phi (x,p_k)$ over this single gap of
$$\left. \Delta \Phi \right|_{2p_{k-1}} \; = \; 1 - \frac{2p_{k-1}}{\mu}.$$ 

The equation
$$ \left. \Delta \Phi \right|_{s} (\cdot, p)  =  j - \frac{|s|}{\mu} $$
shows that the largest positive differences would occur for constellations that are extremely long (large $j$) for their
span (small $|s|$).  Specific extreme examples are provided by the dense admissible $k$-tuples identified by
Engelsma et al. \cite{Eng, EngSuth}.  Relative to other constellations of the same span, the catalog lists the longest 
known admissible examples.

For example, the Engelsma constellation of length $j=459$ and span $\sum g = 3242$ first occurs in $\pgap(\pml{113})$. 
This constellation has average gap-size $7.063$ compared to $\mu(113) = 8.713$, and $\Delta \Phi(x,113)$ rises
by $86.92$ over the course of this constellation.  

Since this constellation is admissible, it occurs among the $p$-rough numbers for all $p \ge 113$.  As $p$ grows, $\mu$ grows, 
and the rise in $\Delta \Phi$ over this constellation increases toward $j=459$.  It approaches its maximum contribution 
of $459$ to $\Delta \Phi$ very slowly.

\begin{figure}[htb]
\centering
\includegraphics[width=5.125in]{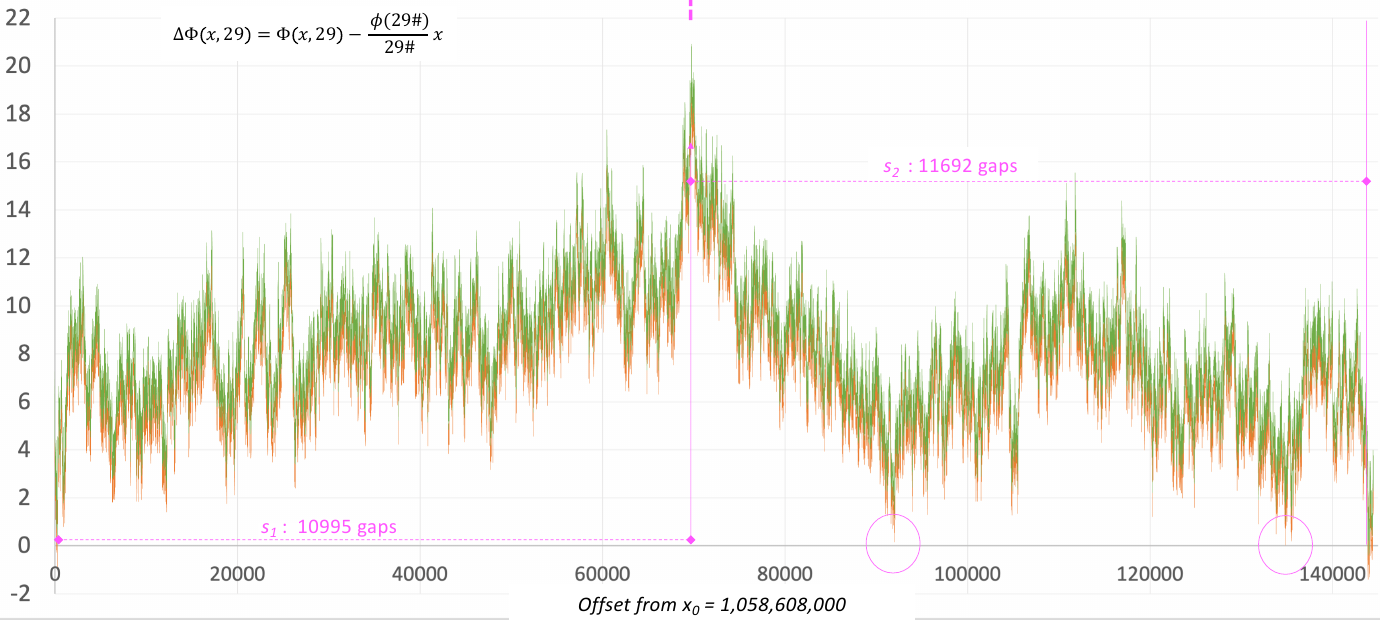}
\caption{\label{29MaxFig} Closeup of $\Delta \Phi(x,29)$ near the the maximum
value $20.9128$. The constellation $s_1$ from $0$ up to the maximum has length $10995$ and the constellation $s_2$ back
down to $0$ has length $11692$.  The $x$-axis values are offset from $x_0=1058608000$.  Note the near-zero values in $s_2$, 
circled in the graph above, of 0.0194 and 0.1968.}
\end{figure}

For $\Delta \Phi$ to reach its maximum value, the constellation from the last zero has to have an average gap size less than
the average $\mu$.  In $\pgap(\pml{29})$ for example, see Figure~\ref{29MaxFig}, the buildup from $x_L~=~1058608248$ to
the peak at $x_{\max}~=~1058677732$ has an average gap size of $6.31960$ over $10995$ gaps, compared to 
$\mu~=~6.33123$ for $p=29$.  

Single gaps and short dense constellations can produce rapid rises, but in order for these to produce maxima for $\Delta \Phi$
the density has to be sustained over long constellations.  The gap $g=2$ produces the largest single rise of $1-\frac{2}{\mu}$.
The constellation $s=242$ produces a rise of $3-\frac{8}{\mu}$.  In Figure~\ref{29MaxFig} we see rapid growth just to the left
of the maximum.  This growth occurs over a constellation of $1992$ gaps of average size $6.09174$.

Sustained rises in $\Delta \Phi(x,p)$ are supported by an abundance of gaps $g < \mu$.  Conversely, the constellations that lead to large drops in the value of $\Delta \Phi$ are supported by an abundance of gaps $g > \mu$.
These are constellations of short length $j$ and relatively large span $|s| = \sum g$.  

The mean $\mu=\frac{\pml{p}}{\phi(\pml{p})}$ grows slowly with $p$.  This has two effects.  First, the demarcation line $g < \mu$
shifts upward.  So the gaps move from contributing negative increments when $g > \mu$ to contributing positive increments
when $g < \mu$.  Figure~\ref{MuFig} shows how slowly the mean gap size $\mu$ grows with $p$.

\begin{figure}[htb]
\centering
\includegraphics[width=5in]{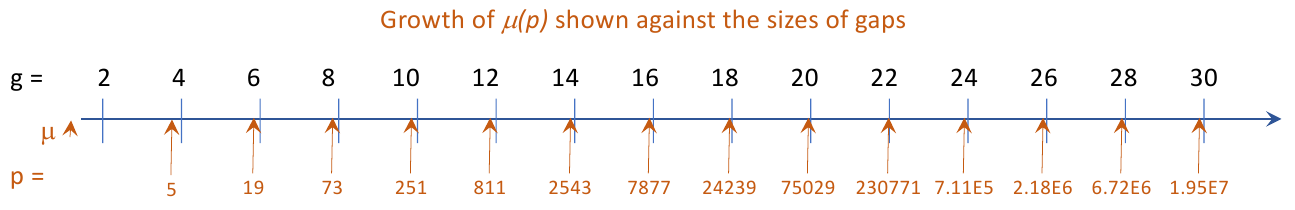}
\caption{\label{MuFig} A diagram of the growth in $\mu(p)$, shown against the sizes of gaps $g$.  Values of $p$ are shown for
each threshold at which the contribution $\left. \Delta \Phi\right|_g = 1 - g/\mu$ from a gap $g$ switches from negative to positive.}
\end{figure}

For peaks and valleys in $\Delta \Phi(x,p)$ we are looking for long constellations to occur between zeroes of the function.
There are two types of zeroes:  those on the downward-sloping segments, and those on the vertical segments of the sawtooth.
As we have seen above, the vertical segments are all of length $1$, and the length of each downward-sloping segment
is proportional to the corresponding gap $g$ in $\pgap(\pml{p})$.  We denote the number of zero-crossings on the downward-sloping 
segments by $N_0^-$ and the number of zero-crossings on the vertical segments by $N_0^+$.  Because the graph of $\Delta \Phi$
(including its vertical segments) is continuous and cyclic, $N_0^- = N_0^+$, and the downward zeroes and vertical zeroes alternate.

There are $\phi(\pml{p})$ downward-sloping segments and $\phi(\pml{p})$ vertical segments in each cycle of $\Delta \Phi(x,p)$.
In Table~\ref{MinMaxTbl} we see that for $p=5$ every vertical edge crosses $0$.  The percentage of vertical edges crossing $0$
drops quickly as $p$ grows.  This leaves more space between zeroes for long constellations that could grow toward peak values
in $\Delta \Phi(x,p)$.

\section{Asymptotics in $p$.}
To understand how the family of functions $\Phi(x,p)$ behaves as $p$ grows, we turn to Merten's Third Theorem
to see the asymptotic attraction to the surrogate line ${y = \frac{e^{-\gamma}}{\ln p} \cdot x}$.
From Merten's Third Theorem 
$$ \frac{\phi(\pml{p})}{\pml{p}} \; = \;  \prod_{q \le p} \left( \frac{q-1}{q} \right) \; \sim \; \frac{e^{-\gamma}+o(1)}{\ln p} , $$
as $p$ gets large.  Although the product does converge to $\frac{e^{-\gamma}}{\ln p}$, the relative error is significant
for a long time.

We return to the decomposition of $\Phi(x,p)$ from Theorem~\ref{ThmPhi},
$$
 \Phi(x,p) = \frac{\phi(\pml{p})}{\pml{p}} \cdot x + \Delta \Phi(x,p)
$$
in which $\Delta \Phi(x,p)$ is periodic and bounded.

Since the waypoints $x_k = k \pml{p}$ and $\tilde{x}_k=\left(k-\frac{1}{2}\right)\pml{p}$ lie on the line of symmetry, 
we have $\Delta \Phi(x_k,p) = \Delta \Phi(\tilde{x}_k,p) = 0$, and
\begin{eqnarray*}
\Phi( x_k, p) & = &  x_k \cdot  \frac{\phi(\pml{p})}{\pml{p}} \biggap  \sim  \biggap \frac{e^{-\gamma} \cdot x_k}{\ln p} \\
\Phi( \tilde{x}_k, p) & = &  \tilde{x}_k \cdot  \frac{\phi(\pml{p})}{\pml{p}}  \biggap  \sim  \biggap \frac{e^{-\gamma} \cdot \tilde{x}_k}{\ln p} 
\end{eqnarray*}
as $p$ gets really large.  For this sequence of waypoints $x_k$ and $\tilde{x}_k$, the asymptotic estimate is as accurate as 
Merten's Third Theorem. 

Other treatments \cite{ChGold, FanPom, FrGrOsc} of asymptotic estimates for $\Phi(x,p)$ factor out the linear dependence as either 
$\frac{\phi(\pml{p})}{\pml{p}}x$ or $\frac{e^{-\gamma}}{\ln p} x$.  In the first approach, which seems more natural,
$$ \Phi(x,p) \biggap = \biggap \left( \frac{\phi(\pml{p})}{\pml{p}} \; x \right) \left[ 1 \; + \; \frac{\Delta \Phi(x,p) \cdot \pml{p}}{\phi(\pml{p}) \cdot x}\right] $$
Since $\Delta \Phi(x,p)$ is bounded, that second term decays as $\frac{1}{x}$.

If we factor out $\frac{x}{\ln p}$ instead, as attributed to Tenenbaum, we have
$$ \Phi(x,p) \biggap = \biggap \frac{x}{\ln p} \;  \left[ \frac{\phi(\pml{p}) \cdot \ln p}{\pml{p}} \; + \; \frac{\Delta \Phi(x,p) \cdot \ln p}{ x}\right] .$$
As $p$ grows, the first term is $e^{-\gamma}$ times the relative error from Merten's Third Theorem, and the second term still decays as $\frac{1}{x}$.


\section{Conclusion}
We have identified a periodic and symmetric structure to $\Phi(x,p)$, tied to the cycle of gaps $\pgap(\pml{p})$, and we have introduced the
derived function
$$ \Delta \Phi(x,p) \gap = \gap \Phi(x,p) - \frac{\phi(\pml{p})}{\pml{p}} \cdot x.$$
This function $\Delta \Phi(x,p)$ measures the deviations of $\Phi(x,p)$ away from its line of symmetry ${y=\frac{\phi(\pml{p})}{\pml{p}} \cdot x}$.

$\Delta \Phi(x,p)$ is periodic and bounded, with period $\pml{p}$.  

There are two sets of waypoints $x_k = k \cdot \pml{p}$ and ${\tilde{x}_k = (k-\frac{1}{2})\pml{p}}$,
for which $\Delta \Phi(x_k,p)= \Delta \Phi(\tilde{x}_k,p) = 0$.
We have the translational symmetry ${\Delta \Phi(x+x_k,p) = \Delta \Phi(x,p)}$ and rotational symmetries around the waypoints.

We have also shown that using any surrogate line as a reference instead of the line of symmetry 
breaks the symmetries and introduces a linear drift.  Since 
$\Delta \Phi(x,p)$ is periodic and bounded, this linear drift eventually swamps any information about the function $\Phi(x,p)$ itself.  Specifically,
the estimates that use $y=\frac{e^{-\gamma}}{\ln p}\cdot x$ as the surrogate line are indirectly studying the error term in Merten's Third Theorem.

The function $\Phi(x,p)$ invites us to use its line of symmetry $y = \frac{\phi(\pml{p})}{\pml{p}}\cdot x$.


\bibliographystyle{amsplain}

\providecommand{\bysame}{\leavevmode\hbox to3em{\hrulefill}\thinspace}
\providecommand{\MR}{\relax\ifhmode\unskip\space\fi MR }
\providecommand{\MRhref}[2]{%
  \href{http://www.ams.org/mathscinet-getitem?mr=#1}{#2}
}
\providecommand{\href}[2]{#2}

\end{document}